\documentclass{amsart}

\usepackage{amsmath}
\usepackage{amssymb}
\def\Zz{\mathbb{Z}} 
\def\Rr{\mathbb{R}} 
\def\Cc{\mathbb{C}} 
\def\NN{\mathcal{N}} 
\def\ker{{\rm ker}\,} 
\def\Hom{{\rm Hom}} 
\def\SW{Stiefel-Whitney\ } 
\def\X{X^m} 
\def\phi{\varphi}
\newtheorem{thm}{Theorem}[section]
\newtheorem{cor}[thm]{Corollary}
\newtheorem{prop}[thm]{Proposition}
\newtheorem{lem}[thm]{Lemma}
\theoremstyle{definition}
\newtheorem{defn}[thm]{Definition}
\theoremstyle{remark}
\newtheorem{rem}[thm]{Remark}
\newtheorem{ex}[thm]{Example}
\begin{document}
%
\title{$\Zz_2$-bordism and the Borsuk-Ulam Theorem} 
\author{M.~C.~Crabb}
\address{Department of Mathematics,\\
University of Aberdeen,\\
Aberdeen, AB24 3UE, UK}
\email{m.crabb@abdn.ac.uk}
\author{D.~L.~Gon\c calves}
\address{Departamento de Matem\'atica \\ 
IME - Universidade de S\~ao Paulo \\ 
Caixa Postal 66281 \\ 
Ag. Cidade de S\~ao Paulo \\S\~ao Paulo,
SP 05314-970, Brazil}
\email{dlgoncal@ime.usp.br}
\author{A.~K.~M.~Libardi}
\address{Departamento de Matem\'atica \\ 
IGCE - UNESP \\ 
Rio Claro, SP 13506-900, Brazil}
\email{alicekml@ms.rc.unesp.br}
\author{P.~L.~Q.~Pergher}
\address{Departamento de Matem\'atica \\ 
Universidade Federal de S\~ao Carlos \\ 
Caixa Postal 676 \\
S\~ao Carlos, SP 13565-905, Brazil}
\email{pergher@dm.ufscar.br}
\thanks{Research partially supported by CNPq and FAPESP-Funda\c c\~ao 
de Amparo a Pesquisa do Estado de S\~ao Paulo, 
Projeto Tem\'atico Topologia Alg\'ebrica, 
Geom\'etrica e Diferencial 2012/24454-8}
\subjclass[2010]{Primary 55M20; Secondary 57R85, 57R75, 55M35}  
\keywords{free involution, Borsuk-Ulam property, line bundle,  
(equivariant) bordism group, Euler class,  \SW numbers}

\begin{abstract}
The purpose of this work is to classify, for given integers $m,\, n\geq 1$,
the bordism class
of a closed smooth $m$-manifold $\X$ with a free smooth involution $\tau$
with respect to the validity of
the {\it Borsuk-Ulam property} that
for every continuous map $\phi : \X \to \Rr^n$
there exists a point $x\in\X$ such that $\phi (x)=\phi (\tau (x))$. 
We will classify a given free $\Zz_2$-bordism class $\alpha$
according to the three possible cases that
(a) all representatives   $(\X , \tau)$ of $\alpha$ 
satisfy the Borsuk-Ulam property; 
\  (b)  there are representatives $(\X_ 1, \tau_1)$ and $(\X_2, \tau_2)$
of $\alpha$ such that $(\X_1, \tau_1)$
satisfies  the Borsuk-Ulam property  but   
$(\X_2, \tau_2)$ does not;
\ (c)  no representative $(\X , \tau)$ of $\alpha$ 
satisfies the Borsuk-Ulam property.
\end{abstract}
\maketitle
\section{ \bf Introduction}
Let $((X,\tau);Y)$ be a triple where $X$ and $Y$ are topological spaces and 
$\tau: X \longrightarrow X$ is a free
involution. We say that $((X,\tau); Y)$  satisfies  
the {\it Borsuk-Ulam property} (which we shall routinely abbreviate to BUP) 
if for every continuous map $\phi :X \longrightarrow Y$ 
there exists a point $x \in X$ such that
$\phi (x)=\phi (\tau(x))$ 
(or, equivalently, we say that  
the Borsuk-Ulam property holds for $((X,\tau); Y)$). 
We also say that the pair $(X,\tau)$ satisfies the Borsuk-Ulam property 
with respect to maps into $Y$.

Let $S^m$ be the $m$-dimensional sphere, $A:S^m \to S^m$ 
the antipodal involution and $\Rr^n$ the
$n$-dimensional Euclidean space. 
The famous Borsuk-Ulam Theorem states that, if $\phi : S^m \to \Rr^m$
is any continuous map, then there exists a point $x\in  S^m$ 
such that $\phi (x)=\phi (A(x))$ (see \cite{Borsuk}). We remark
that, as an easy consequence of the Borsuk-Ulam Theorem, 
$((S^m,A); \Rr^n)$ satisfies the Borsuk-Ulam property for $n \leq m$,
but the BUP does not hold for $((S^m,A); \Rr^n)$ if $n>m$, because $S^m$
embeds in $\Rr^n$. 

After the comments above a natural generalization of the 
Borsuk-Ulam Theorem consists in replacing
$S^m$ by a connected, closed $m$-dimensional smooth manifold $\X$ and $A$
by  a free smooth involution $\tau$ defined on $\X$, 
and then, considering continuous maps from $\X$ into $\Rr^n$ ($n>0$), 
to ask which triples $((\X ,\tau); \Rr^n)$
satisfy the  Borsuk-Ulam  property (BUP). It is known
that if $n>m$, the BUP does not hold (see \cite{GHZ}).  
This leaves the interesting cases where $n\leq m$ and $m>1$. 
For $n=2$ the problem is equivalent to an algebraic problem involving the 
fundamental group;
see, for example, \cite{GHZ} 
or Lemmas \ref{Lemma2.4} and \ref{dim2}.
For $n=m\geq 2$ the problem is well understood in terms
of the cohomology of the orbit space, as we explain below.

A major problem is to find the greatest $n\leq m$ such that the BUP 
holds for a specific $(\X ,\tau)$.
Since for $m>1$ and $n=1$ the BUP always holds, this greatest integer  
will be greater than or equal to $1$.  This is
equivalent to finding all values $n$ for which the BUP  
holds for  $((\X ,\tau); \Rr^n)$.
The greatest integer $n$ for which the BUP holds for  
the triples above is also known in the literature
as the {\it index} of $\tau$ on $\X$.
Unfortunately, at this level of generality the problem seems too hard.  
In this paper we shall tackle the problem up to bordism.

Before describing our results,
we fix some notation that will be used throughout the paper.
The basic object of interest is a pair $(\X ,\tau )$ consisting of
a closed smooth manifold $\X$ of dimension $m\geq 1$ and a smooth free
involution $\tau :\X \to\X$. 
We do not always require $\X$ to be connected.
The orbit space $M=\X /\tau$ is a closed smooth $m$-manifold.
We shall write $\lambda$ for the orthogonal real line bundle over $M$ 
associated with the double cover $\X \to \X /\tau$ and
$f: M\to BO (1)$ for its classifying map.

Let us begin with the cohomological criterion from \cite[Theorem 3.4]{GHZ} 
for the case $n=m$.  
\begin{thm}
Consider a pair $(\X ,\tau )$ with corresponding real line bundle
$\lambda$ over $M=\X /\tau$. 
Then the Borsuk-Ulam property holds for $((\X ,\tau); \Rr^m)$ if
and only if the $m$-fold cup product  $w_1(\lambda )^m$ is non-zero.
\end{thm}
Recently Musin \cite{Musin} showed how to characterize
the Borsuk-Ulam property in terms of  $\Zz_2$-bordism classes. 
Referring to Section 2 for more details,
we denote the group of bordism classes of $m$-dimensional pairs
$(\X ,\tau )$ by $\NN_m(\Zz_2)$, writing 
$[\X , \tau]$ for the $\Zz_2$-bordism class of $(\X , \tau)$.
We recall from \cite[p.~74]{BtD}
that $\NN_*(\Zz_2)$ is a free graded module
over the unoriented bordism ring $\NN_*$
with one generator $p_i= [S^i, A]$ in each dimension $i\geq 0$, 
where $A$ is the antipodal involution on the sphere $S^i$.
Musin's result \cite[Theorem 2]{Musin},
for smooth rather than PL manifolds, can be formulated as follows.
\begin{thm} 
Let  $\alpha \in \NN_m(\Zz_2)$, where $m > 1$, be written as
$\alpha =a_0p_m+ a_1p_{m-1}+\ldots +a_{m-1}p_1+a_mp_0$ 
where $a_i\in \NN_{i}$. 
Then
\par\noindent
\hbox to\parindent{\rm (i)\hfil}
For any representative $(\X ,\tau )$ of $\alpha$, the coefficient
$a_0$ is equal to the characteristic number $w_1(\lambda )^m[M]$
of the associated line bundle $\lambda$ over $M=\X /\tau$.
\par\noindent
\hbox to\parindent{\rm (ii)\hfil}
There is a representative $(\X ,\tau )$ of $\alpha$ with $\X$ connected.
\par\noindent
\hbox to\parindent{\rm (iii)\hfil}
If $a_0=1$, then for every representative $(\X ,\tau )$ of $\alpha$
the BUP holds for $((\X ,\tau ); \Rr^m)$.
\par\noindent
\hbox to\parindent{\rm (iv)\hfil}
If $a_0=0$, then there is no representative $(\X ,\tau )$ of $\alpha$ with
$\X$ connected such that the BUP holds for $((\X ,\tau ); \Rr^m)$.
\end{thm} 
A short proof, using quite different
techniques from those employed in \cite{Musin}
will be given in Section 3. 
We are interested in generalizing the result when
$\Rr^m$ is replaced by $\Rr^n$ for $n\leq m$.
More precisely we consider the following question.

{\bf Question}: Given 
a free $\Zz_2$-bordism class $\alpha\in \NN_m(\Zz_2)$ and an integer  $n$, 
can we find a representative $(\X ,\tau)$ of $\alpha$ such
that the BUP holds for $((\X ,\tau); \Rr^n)$ ?
\begin{ex}
Let $m>1$.  Consider the connected manifold $\X =S^1\times S^{m-1}$
and the two involutions $\tau_1 = 1\times A$ and $\tau_2=A\times 1$ on
$\X$.
Both $(\X ,\tau_1)$, the boundary of $D^2\times S^{m-1}$,
and $(\X ,\tau_2)$, the boundary of $S^1\times D^m$, represent
$0$ in $\NN_m(\Zz_2)$.
The BUP holds for $((\X ,\tau_1); \Rr^n)$ for $1\leq n<m$,
but for $1<n\leq m$ the BUP does not hold for $((\X ,\tau_2); \Rr^n)$. 
\end{ex}

This example shows that,
given $\alpha$ and an integer $n$, it is not true, in general,
that if for one representative  $(\X_1,\tau_1)$ of $\alpha$ the BUP holds  
for $((\X_1,\tau_1); \Rr^n)$ then
it also holds for every other representative $((\X_2,\tau_2); \Rr^n)$.
For each $m$-dimensional free $\Zz_2$-bordism class $\alpha$ and 
integer $n$, one would, therefore, like
to  decide whether
the BUP with respect to $\Rr^n$ holds  
(a) for all of the representatives $(\X ,\tau)$ of $\alpha$,
or (b) for some, but not all, of the representatives, or (c) for none of them.

Our main results are as follows.
\begin{thm} 
Let  $\alpha \in \NN_m(\Zz_2)$, where $m>1$, be written as
$\alpha =a_0p_m+ a_1p_{m-1}+\ldots +a_{m-1}p_1+a_mp_0$ 
where $a_i\in \NN_{i}$. Suppose that $1<n<m$.
Then:

\par\noindent
\hbox to\parindent{\rm (i)\hfil}
There is a representative $(\X ,\tau )$ of $\alpha$, with $\X$ connected,
such that the BUP holds for $((\X ,\tau );\Rr^n)$.

\par\noindent
\hbox to\parindent{\rm (ii)\hfil}
If $a_k\not=0$ for some $k\leq m-n$, then for every representative
$(\X ,\tau )$ of $\alpha$ the BUP holds for $((\X ,\tau ); \Rr^n)$.

\par\noindent
\hbox to\parindent{\rm (iii)\hfil}
If $a_k=0$ for all $k\leq m-n$, then there is some representative
$(\X ,\tau )$ of 
$\alpha$, with $\X$ connected, such that the BUP does not hold
for $((\X ,\tau ); \Rr^n)$.
\end{thm}
\begin{thm} 
Consider a pair $(\X ,\tau )$ and corresponding real line bundle
$\lambda$ over $M=\X /\tau$. Write
$[\X , \tau ]=a_0p_m+\ldots +a_mp_0$. Then, for $n\leq m$,
we have $a_k=0$ for all $k\leq m-n$ if and only if 
$$ 
w_{j_1}(M) \cdots w_{j_s}(M)w_1(\lambda)^{m-\Sigma j_i}[M]=0 
$$
for all  $j_1, \cdots ,j_s\geq 1$, $s\geq 0$, with
$\Sigma_{i=1}^{s} j_i\leq m-n$.
\end{thm}

\begin{cor} 
If $a_k\not=0$ for some $k\leq m-n$, then $w_1(\lambda )^n\not=0$.
\end{cor}
We remark that the stronger problem  of  classifying all triples 
$((\X , \tau);  \Rr^n)$,  where
$(\X , \tau)$ runs over the set of all representatives in a given 
$\Zz_2$-bordism
class for which the BUP holds is much harder and closely
related to the problem of finding the index
(as defined above) of a pair  $(\X ,\tau)$.

The following more refined classification question is open and certainly 
far more
difficult than the classification question which has been considered in the 
present work.
More precisely,  let  $\alpha \in \NN_m(\Zz_2)$ such that $\alpha$ is 
of the form
$a_0p_n+ a_1p_{n-1}+\ldots +a_{n-1}p_1+a_kp_0$ 
where $a_i\in \NN_{n-i}$ and $a_0\ne 0$. 
Suppose  $n<m$ (otherwise the main question  is known). 
Given $\alpha\in \NN_m(\Zz_2)$ first enumerate the classes of 
equivalent classes of pairs in the sense of \cite[Corollary 2.3]{GHZ} 
and  a main question is to
determine  the involution classes  $[\X , \tau ]$ for which the triples
 $((\X , \tau); \Rr^n) $  satisfy the BUP as  $(\X , \tau)$ 
runs over the  elements of $\alpha$.
In particular find such a classification for   
$\alpha\in \NN_m(BO (1))$ of the form $ap_n$ where $a\in \NN_{m-n}$.
\section{Preliminaries}
We begin by  reviewing some basic facts about free $\Zz_2$-bordism.
\begin{defn} \label{Definition2.1} 
Two free involution pairs $(\X ,\tau )$ and $({X'}^m,\tau')$ 
are cobordant if there
exists a smooth compact $(m+1)$-dimensional manifold $W^{m+1}$, equipped with a free smooth involution 
$\upsilon :W^{m+1} \longrightarrow W^{m+1}$ such that the boundary of $W^{m+1}$ 
is the disjoint union $\X \cup {X'}^m$, $\upsilon |_{\X } = \tau$ and 
$\upsilon |_{{X'}^m}=\tau'$. 
This gives rise to a set of cobordism classes which we denote by $\NN_m(\Zz_2)$.
\end{defn}
Recall that associated with a smooth free involution $\tau :\X \to \X$
we have a closed smooth $m$-dimensional manifold $M=\X /\tau$ and
an orthogonal real line bundle $\lambda$ over $M$ classified by a map
$f: M\to BO (1)$.
The total space of $\lambda$ is
$(\X \times  \Rr )/\sim$, where $\sim$ identifies $(x,t)$ to $(\tau (x),-t)$.
Conversely, from any orthogonal real line bundle $\lambda$ over a closed smooth
$m$-dimensional manifold $M$ one obtains an associated 
smooth free involution by taking $\X$ to be the total space
of the unit sphere bundle $S(\lambda)  \to M$, 
equipped with the fibrewise involution
$\tau =-1 : S(\lambda) \to S(\lambda)$ that
interchanges the points of each fibre. 
From \cite[20-4; p.~71]{CF} we have:
\begin{prop} \label{Theorem2.2} 
The correspondence $[\X , \tau ] \mapsto [M, f]$
described above defines an isomorphism of $\NN_*$-modules between 
$\NN_*(\Zz_2)$ and $\NN_*(BO (1))$.
\end{prop}
Using this correspondence, we shall work primarily with a
pair $(M,f)$ (and the associated real line bundle $\lambda$)
representing a class in $\NN_m(BO (1))$
and then translate results into statements about the free $\Zz_2$-manifold
$(\X ,\tau )$ representing a class in $\NN_m(\Zz_2)$.
The Hopf line bundle over $BO (1)=\Rr P^\infty$ will be denoted by
$H$. Thus the line bundle $\lambda$ is isomorphic to $f^*H$.

We begin with the Borsuk-Ulam property for $((\X ,\tau );\,\Rr^n)$.
\begin{lem} \label{Lemma2.3}  
Consider a line bundle $\lambda$ over a closed $m$-manifold $M$.
Then there is
a map $\phi: S(\lambda) \longrightarrow \Rr^n$ such that 
$\phi(-v) \neq \phi(v)$ for all $ v \in S(\lambda)$
if and only if the
bundle $n\lambda =\Rr^n\otimes\lambda$ over $M$ has a nowhere zero section. 
\end{lem}
\begin{proof}
Given such a map $\phi$, we can write down a nowhere zero section
mapping $x \in M$ to
$(\phi(v)-\phi(-v))\otimes v\in \Rr^n\otimes\lambda_x$ for $v\in S(\lambda_x)$.
Conversely,
a nowhere zero section $s$ can be expressed as $s(x)=\phi (v)\otimes v$,
$v\in S(\lambda_x)$, with $\phi (-v)=-\phi (v)\not=0$.
\end{proof}
The case in which $n$ is equal to the dimension $m$ of $M$ is 
completely settled by cohomology.
\begin{lem} \label{Lemma2.4} 
Let $M$ be a closed smooth $m$-manifold 
and let $\lambda$ be a real line 
bundle over $M$. Then $m\lambda=\Rr^m\otimes\lambda$
admits a nowhere zero section if and only if $w_1(\lambda)^m = 0$.
\end{lem}
\begin{proof}
We may assume that $M$ is connected.
The bundle $\xi = m\lambda$ admits a nowhere zero 
section if and only if the Euler class
$e(\xi) \in H^m(M; \Zz (\xi))$ in the cohomology group of $M$ with integer 
coefficients twisted by the orientation bundle of $\xi$ is zero. But
$2e(\xi) =0$, since $e(\xi)= e(\lambda)^m$ and $ 2e(\lambda) = 0$. 
By Poincar\'{e} duality, $H^m(M; \Zz (\xi))$ is isomorphic
to either $\Zz$ or $\Zz /2\Zz.$  In the first case, we must have $e(\xi)=0$,
and then $w_1(\lambda )^m=0$, too. 
In the second case, reduction mod $2$:
$H^m(M; \Zz (\xi)) \longrightarrow H^m(M; \Zz / 2\Zz )$ is an isomorphism 
and  maps $e(\xi)$ to $w_1(\lambda)^m$.
\end{proof}
\begin{proof}[Proof of Theorem 1.1]
The theorem follows at once from Lemma \ref{Lemma2.3} and Lemma \ref{Lemma2.4}.
\end{proof}
The cases in which $n=1$ or $n=2$ are also easily resolved by
cohomology calculations.
\begin{lem} \label{dim2}
Let $\lambda$ be a real line bundle over a closed manifold $M$.
Consider the cases
{\rm (i)} $n=1$,  
{\rm (ii)} $n=2$.  Then $n\lambda$ admits a nowhere zero section
if and only if
{\rm (i)} $w_1(\lambda )=0$, 
{\rm (ii)} $w_1(\lambda )\in H^1(M;\,\Zz /2\Zz )$ is
the reduction mod $2$ of an integral class in $H^1(M;\,\Zz )$,
respectively. 
\end{lem}
\begin{proof}
(i).\ 
The real line bundle $\lambda$ is classified by $w_1(\lambda )\in H^1(M;\,
\Zz /2\Zz )$ and admits a nowhere zero section if and only if it is trivial. 

(ii).\ 
The complex line bundle $\Cc\otimes\lambda$, isomorphic to
$2\lambda$ as a real bundle, is similarly classified by 
the first Chern class $c_1(\Cc \otimes \lambda )\in H^2(M;\,\Zz )$
and admits a nowhere zero section if and only if it is trivial. 
Now $c_1(\Cc\otimes \lambda )=\beta w_1(\lambda )$,
where $\beta$ is the Bockstein homomorphism in the exact sequence:
$$
H^1(M;\,\Zz ) \xrightarrow{2} H^1(M;\,\Zz ) \to H^1(M;\, \Zz /2\Zz )
\xrightarrow{\beta} H^2(M;\,\Zz ),
$$
as follows from the observation that $\beta : H^1(BO (1);\, \Zz /2\Zz ) \to
H^2(BO (1);\, \Zz )$ is an isomorphism.
Hence $\beta w_1(\lambda )$ is zero if and only if $w_1(\lambda )$ is
integral.
\end{proof}
\section{Connectedness}
Let $\alpha\in \NN_m(BO (1))$ be a $\Zz_2$-bordism class. 
In this section we first show that, if
$m>1$, there is a representative $(M,f)$ of $\alpha$ such that the
associated $\Zz_2$-manifold $\X = S(\lambda )$ is connected.
For  $n<m$, there is a representative $(M , f)$ of $\alpha$, with
$\X$ connected,  for which the BUP holds for the triple 
$((\X, \tau); \Rr^n))$.
\begin{prop} \label{Proposition3.1} 
If $m > 1$, then every class in $\NN_m (BO (1))$ has a representative
$(M', f')$ with $M'$ connected and admitting an embedding 
$\Rr P^{m-1} \hookrightarrow M'$
such that $\lambda' =(f')^*H$ restricts to the Hopf bundle over $\Rr P^{m-1}$. 
It follows that $S(\lambda' )$ is connected and that,
if $n < m$, every section of
$n\lambda'$ has a zero {\rm (}lying in the submanifold 
$\Rr P^n\subseteq  \Rr P^{m-1}\subseteq M${\rm )}.
\end{prop}
\begin{proof}
This may be proved by an elementary surgery argument as follows.
Consider a class $(M,f)$. By adding, if necessary, 
the disjoint union of two copies of $\Rr P^m$
and the classifying map $h: \Rr P^m \to   BO (1)$ of the Hopf line bundle, 
we may assume
that at least one component of $(M,f)$ is $(\Rr P^m , h)$. 
We now use the connected sum to reduce the number of components one at a time,
connecting in one of the components $(\Rr P^m , h)$ at the last step
to produce $(M',f')$ cobordant to $(M,f)$ with $M'$ connected.
By construction, $M'$ contains an embedded submanifold $\Rr P^{m-1}$
such that $(f')^*H$ restricts to the Hopf bundle.
The condition that $m$ is strictly greater than $1$,
so that $m-1\geq 1$, then guarantees that $\lambda'$ is non-trivial.
\end{proof}
\begin{rem}
The case $m=1$ is exceptional. The group $\NN_1(BO (1))$ has
two elements $0$ and $p_1$.
The class $p_1$ is represented by the manifold $M=\Rr P^1$ and
the classifying map of the Hopf line bundle $\lambda =H$, 
and $S(\lambda )= S^1$ is connected.
The only connected representative of $0$ is the circle $M=S^1$
with the trivial bundle $\lambda$, but in that case $\X =S(\lambda )$ is not
connected.
\end{rem}
\begin{prop} \label{Corollary2.6}  
A map $(M,f)$ is cobordant in $\NN_m(BO (1))$ to a map $(M',f')$ 
such
that $(f')^*(m H)$ has a nowhere zero section if and only if 
$$
w_1(f^* H)^m[M] = 0.
$$
\end{prop}
\begin{proof}
We may assume, by Proposition \ref{Proposition3.1}, that
$M$ is connected, and then the bordism invariant $w_1(\lambda )^m[M]$ is
zero if and only if $w_1(\lambda )^m=0$. Now apply Lemma \ref{Lemma2.4}.
\end{proof}
\begin{proof}[Proof of Theorem 1.2]
To see that the bordism invariant $w_1(\lambda )^m[M]$ is
equal to $a_0$, it is enough to check when $M$ has the form 
$N\times \Rr P^{m-i}$ for some manifold $N$ of dimension $i$ and $\lambda$
is the pullback of the Hopf bundle over $\Rr P^{m-i}$. 
Then $w_1(\lambda )^m=0$ if $i>0$.

The assertion (ii) is immediate from Proposition \ref{Proposition3.1} which
provides a representative $(M,f)$ with $M$ connected and $\lambda$ non-trivial,
so that $\X =S(\lambda )$ is also connected.

If $a_0=1$, we must have $w_1(\lambda )^m\not=0$,
and, if $a_0=0$ and $M$ is connected, then $w_1(\lambda )^m=0$.
Hence, parts (iii) and (iv) follow from Theorem 1.1.
\end{proof}
\begin{prop}\label{Proposition3.3}
For $1<n\leq m$,
suppose that $(M,f)$ has the property that $n\lambda$ admits a nowhere
zero section. Then there is a map $(M',f')$ cobordant to
$(M,f)$ in $\NN_m(BO (1))$ such that $n\lambda '$ admits a nowhere zero
section, $M'$ is connected and $\lambda'$ is non-trivial.
\end{prop}
\begin{proof}
Let $h : \Rr P^1 \to BO (1)$ classify the Hopf bundle $H$ over the projective
line.
Then $S^{m-1}\times (\Rr P^1,h)$ is the boundary of $D^m\times (\Rr P^1,h)$
and $n(h^*H)$ admits a nowhere zero section.
We can now argue as in the proof of Proposition \ref{Proposition3.1}
to construct $(M',f')$ cobordant to $(M,f)$ and containing an embedded
$\Rr P^1$ on which $\lambda'$ restricts to $H$.
\end{proof}
\section{The cobordism Euler class and the BUP for higher codimension}
We begin with some remarks on the cobordism Euler class.
Consider an $n$-dimensional real vector bundle $\xi$ over
the closed manifold $M$ of dimension $m$.
We recall that $\xi$ has a canonical cobordism 
Thom class $u_\NN (\xi )\in \NN^n(D(\xi), S(\xi)))$ in the cobordism
of the disc modulo the sphere bundle;
see, for example, \cite{BtD}.
The {\it cobordism Euler class} $e_\NN (\xi )\in \NN^n(M)$ is the
restriction $z^*u_\NN (\xi )$ to the zero-section $z$ of $\xi$.

Now suppose that the vector bundle $\xi$ is smooth and consider
a smooth section $s$ of $\xi$ that is transverse to the zero-section $z$.
The zero-set $Z$ is then a smooth submanifold 
of $M$ of dimension $m-n$ with its normal bundle isomorphic to $\xi$. 
This manifold $Z\hookrightarrow M$ represents a bordism class 
$[Z]\in\NN_{m-n}(M)$,
which corresponds under the duality isomorphism $\NN_{m-n}(M)=\NN^n(M)$
to the Euler class $e_\NN (\xi )$.
\begin{rem}\label{euler}
In mod $2$ homology, $Z$ represents an element of $H_{m-n}(M)$
dual to the cohomology Euler class $w_n(\xi )\in H^n(M)$.
This means that,
for a class $x\in H^{m-n}(M)$ restricting to $x|Z\in H^{m-n}(Z)$, 
we have $(x\cdot w_n(\xi ))[M]= (x|Z) [Z]\in \Zz /2\Zz$,
where here $[Z]$ and $[M]$ are the respective fundamental classes.
\end{rem}

We now specialize to the case of interest in which
we have a map $f : M \to  BO (1)$ classifying a line bundle $\lambda= f^*H$
and take $\xi =n\lambda$. 
\begin{defn}
The submanifold $Z\subseteq M$ described above
represents a class $[Z]\in\NN_{m-n}(M)$, dual to
the Euler class $e_{\NN}(\lambda)^n = e_{\NN}(n\lambda)\in \NN^n(M)$.
Its image $f_*[Z]\in \NN_{m-n} (BO (1))$ under $f$ will be denoted
by $e_n(M,f)$.
\end{defn}
\begin{lem}\label{bordisminv}
The class $e_n(M,f)\in\NN_{m-n}(BO (1))$ depends only on the bordism
class $[M,f]\in\NN_m(BO (1))$.
\end{lem}
\begin{proof}
Suppose that $M=\partial W$ is the boundary of a compact manifold $W$
and that $f: M\to BO (1)$ extends to a map $h : W\to BO (1)$.
We may choose a smooth section of $h^*(nH)$ over $W$ which is
transverse to the zero-section and such that its zero-set $V$ intersects
$M$ transversely in its boundary $Z=\partial V$.
Then $(V,h|V)$ bounds $(Z,f|Z)$, and hence $f_*[Z]=0$.

See, also, the proof of Lemma \ref{Lemma4.4}.
\end{proof}
If $n\lambda$ admits a nowhere zero section, then $e_\NN (\lambda )^n=0$
and hence $e_n(M,f)=0$. Arguments going back to the work of Koschorke 
\cite{Koschorke} show that the vanishing of $e_n(M,f)$ is sufficient
for the existence of a nowhere zero section up to cobordism.
\begin{prop} \label{Proposition4.3} 
Suppose that $1\leq n\leq m$.
A map $(M, f )$ is cobordant in $\NN_m (BO (1))$ 
to a map $(M' , f')$
such that the multiple $n\lambda'$ of $\lambda'= (f')^*(H)$ 
has a nowhere zero section 
if and only if $e_n(M, f ) = 0$.
\end{prop}
\begin{proof}
Retaining the notation used in the construction of $e_n(M,f)$,
suppose that $e_n(M, f ) = 0$.
Then we have a compact manifold $V$ of dimension $m-n+1$ with 
boundary $\partial V = Z$ and a map
$g : V\to   BO (1)$ extending $f|Z$. 
Choose an inner product on the normal bundle $\nu$ of the embedding
$Z\hookrightarrow M$ and a tubular neighbourhood $D(\nu)\hookrightarrow M$.
Over the closed disc bundle $D(\nu)$ the bundle $n\lambda$ is identified 
(up to homotopy) with the pullback of $\nu$ from $Z$ and 
the section $s$ corresponds to the diagonal (or identity) section. 

Now the boundary $\partial(M - B(\nu))$
of the complement of the open disc bundle $B(\nu )$ 
is the sphere-bundle $S(\nu)=S(n\lambda | Z )$. 
This is the same as the boundary $\partial S(g^*(nH)) = \partial S(f^*(nH)|Z)$
of the sphere bundle of $g^*(nH)$ over $V$. 
We can thus perform surgery to construct the manifold
$$   
M' = (M - B(\nu))\cup_{S(\nu)}S (g^*(nH)) 
$$
and extend $f$ on $M - B(\nu)$ to a map $f' : M'\to BO (1)$ by using $g$. 
A nowhere zero section $s'$ of $n\lambda'$ is obtained
by gluing $s$ on $M - B(\nu)$ and the tautological identity section
of $S(n\lambda')$ over $S(g^*(nH))$. 

The manifolds
$M$ and $M'$ are cobordant by the bordism
$$               
(M\times [0,1])\cup_{D(\nu )\times \{ 1\}}(D(g^*(nH))\times\{1\})
$$
(after smoothing the corners)
with boundary $(M\times\{0\})\cup (M'\times\{ 1\})$.
\end{proof}
We recall in the next lemma the calculation of the bordism
of $\Rr P^\infty$.
\begin{lem} \label{Lemma4.4} 
Let $p_i\in \NN_m(BO (1))$ be the class
represented by the real projective space $\Rr P^i$ and
the classifying map $\Rr P^i\to BO (1)$ of the Hopf bundle.
Then
$$
\NN_m(BO (1))= \bigoplus_{0\leq i\leq m}\NN_{m-i}\, p_i.
$$
The map $e_n : [M, f ] \mapsto e_n(M, f)$ fits 
into a short exact sequence:
$$  
0 \to \NN_m(P(\Rr^n)) \to \NN_m(BO (1))\xrightarrow{e_n} 
\NN_{m-n}(BO (1)) \to   0
$$
in which 
the homomorphism $e_n : \NN_{*}(BO (1)) \to  \NN_{*-n}(BO (1))$  
is $\NN_*$-linear and 
$$
\text{$e_n(p_m) =p_{m-n}$ if $m \geq  n,$ $e_n(p_m)=0$ if $m < n$.}
$$
\end{lem}
\begin{proof}[Some remarks on the proof]
By the construction of $e_n(M, f)$, the homomorphism $e_n$ is 
multiplication by
the cobordism Euler class $e_{\NN}(H)^n $. 
We can realize $\NN_m(BO (1))$  as $\NN_m (P(\Rr^N))$
for large $N$ and then the sequence is the exact sequence of the pair 
$(P(\Rr^N), P(\Rr^n))$:
$$    
\NN_m(P(\Rr^n)) \to  \NN_m(P(\Rr^N)) \to
 \NN_m(P(\Rr^N), P({\Rr^n})) = \NN_{m-n}(P(\Rr^{N -n} )).
$$
The calculation of $e_n(p_m)$ can be made by applying the construction of 
$e_n(M, f)$
to $M = \Rr P^m$ with $f$ classifying the Hopf line bundle.
One has a section of $n\lambda$ with zero-set $\Rr P^{m-n}$ if $m\geq n$
and empty zero-set if $m<n$.
\end{proof}
We can now state two corollaries of Proposition \ref{Proposition4.3}.
\begin{cor} \label{Corollary4.5} 
{\rm (Compare \cite[Proposition 2.2]{GHZ}.)}
The class $e_n(M, f)$ vanishes if and only if $(M, f)$ is cobordant to
some $(M', f')$ for which the map $f'$ factors through 
the classifying map $P(\Rr^n ) = \Rr P^{n-1} \to BO (1)$.  
\qed
\end{cor}
\begin{cor}
Given $(M,f)$, the coefficient $a_n\in\NN_{m-n}$ 
in the expansion of the bordism class 
$[M,f]$ as $a_0p_m+\ldots +a_mp_0$
is represented by the zero-set of a generic smooth section of $n\lambda$.
\end{cor}
\begin{proof}
The class $e_n(N,f)$ is represented by the zero-set $Z$ of a generic
section of $n\lambda$ with a classifying map $Z\to BO (1)$.
Forgetting the classifying map, we have $a_n=[Z]$. 
\end{proof}
\begin{proof}[Proof of Theorem 1.4]
Part (i) follows at once from Proposition \ref{Proposition3.1}.

For (ii), suppose that $a_k\not=0$ where $k\leq m-n$.
Then $e_n(M,f)\not=0$, by Lemma \ref{Lemma4.4},
and so,
by Proposition \ref{Proposition4.3}, every section of $n\lambda$ has
a zero.
The assertion follows from Lemma \ref{Lemma2.3}.

In part (iii), we are given that $e_n(M,f)=0$. 
By Proposition \ref{Proposition4.3},
we may assume that $n\lambda$ has a nowhere zero section,
and then, because $n>1$, that $M$ is connected and $\lambda$ is non-trivial,
by Proposition \ref{Proposition3.3}.
\end{proof}
\section{The BUP for higher codimension and \SW classes}
The bordism class of $(M,f)$ in $\NN_m(BO (1))$ 
is determined by
the collection of \SW numbers 
$$
w_{j_1}(M)w_{j_2}(M)\cdots
w_{j_s}(M)w_1(\lambda )^{m-\sum j_i} [ M],
$$
where $s\geq 0$, $j_1,\, \ldots ,\, j_s\geq 1$ and $\sum j_i\leq m$.
Our main result in this section
is a criterion for the vanishing of $e_n(M, f)$ in terms of
these characteristic numbers.
\begin{thm} \label{Theorem5.1} 
The class $e_n(M, f)$ vanishes if and only if
$$ 
w_{j_1}(M) \cdots w_{j_s}(M)w_1(\lambda)^{m-\Sigma j_i}[M]=0 
$$
for all  $j_1, \ldots ,j_s\geq 1$, $s\geq 0$, with
$\Sigma_{i=1}^{s} j_i\leq m-n$ {\rm (}that is, $m-\Sigma j_i\geq n${\rm )}.
\end{thm}
\begin{proof}
The result is clearly true if $n>m$. 
If $n=m$, it reduces to $w_1(\lambda)^m[M]=0$.
For $n < m$, we look at the construction of $e_n(M, f)$ 
in terms of the zero-set $Z$. 
Now $e_n(M,f)\in\NN_{m-n}(BO (1))$ is zero if and only if
all the characteristic numbers
$w_{j_1}(Z) \cdots w_{j_s}(Z)w_1(\lambda)^{m-n-\Sigma j_i}[Z]$
with $\Sigma j_i\leq m-n$ vanish.
We must show that this condition is equivalent to the vanishing of 
all the $w_{j_1}(M) \cdots w_{j_s}(M)w_1(\lambda)^{m-\Sigma j_i}[M]$
with $\Sigma j_i\leq m-n$.
But this follows from the congruence
$$
w_{j_1}(Z) \cdots  w_{j_s}(Z)w_1(\lambda)^{m-n-\Sigma j_i}[Z]
$$
$$ = w_{j_1}(TM-n\lambda) \cdots
w_{j_s}(TM-n\lambda) w_1(\lambda)^{m-n-\Sigma j_i} e(n\lambda)[M]
$$
$$
\equiv 
w_{j_1}(M) \cdots w_{j_s}(M)w_1(\lambda)^{m-\Sigma j_i}[M] 
$$
modulo terms involving $w_1(\lambda)^{n'}$ for $n' > n$.
(The first equality is a special case of Remark \ref{euler}.)
\end{proof}
\begin{proof}[Proof of Theorem 1.5 and Corollary 1.6]
We have $a_k=0$ for all $k\leq m-n$ if and only if $e_n(M,f)=0$,
by Lemma \ref{Lemma4.4}.
The result of Theorem 1.5 is then immediate from Theorem \ref{Theorem5.1}.
Corollary 1.6 follows from the observation that $e_n(M,f)=0$
if $w_1(\lambda )^n=0$.
\end{proof}
\begin{cor} \label{Corollary5.2} 
If $w_1(\lambda )^n=0$, then $e_n(M,f)=0$.
But for $m>1$, if $n<m$, there is always a pair $(M',f')$ cobordant to $(M,f)$
such that $w_1(\lambda')^n\not=0$.
\end{cor}
\begin{proof}
If $w_1(\lambda )^n=0$, then all the characteristic numbers appearing in
Theorem \ref{Theorem5.1} are zero and, hence, $e_n(M,f)=0$.
For $m>1$,
Proposition \ref{Proposition3.1} supplies $(M',f')$ with $w_1(\lambda')^{m-1}
\not=0$.
\end{proof}
\section{Higher dimensional trivial summands}
In this section there is interest even when $n$ is greater than
the dimension $m$ of the manifold.

Having formulated our basic question in terms of sections of the
vector bundle $n\lambda$ over $M$,
one can go on to ask, for each $r > 1$, whether $(M,f)$ is cobordant to 
some $(M',f')$ such that $n \lambda'$ admits $r$ linearly
independent sections, that is, a trivial subbundle of dimension $r$. 
There is a similar criterion in a certain metastable range. 

Recall that the space $A_k(\Rr^r,\Rr^n)$
of linear maps $\Rr^r\to\Rr^n$ with kernel of dimension
$k$ is a manifold of dimension $(n+k)(r-k)=rn-k(n-r+k)$.

Consider the vector bundle $\Hom (\Rr ^r, n\lambda)$, with
fibre at $x \in M$ the space of linear maps 
$v:\Rr^r\to\Rr^n \otimes \lambda_x$,
and let $A_1(\Rr^r, n\lambda )$ be the subbundle with fibre
at $x$ the space of maps $v$ with $\dim\ker v=1$.

Suppose that $m<2(n-r+2)$. 
This dimensional restriction ensures that a generic smooth section 
$s$ of  $\Hom (\Rr^r, n\lambda )$ is transverse to $A_1(\Rr^r,n\lambda )$
and has the property that $\dim \ker s_x \leq 1$ for all $x \in M$.
The subspace 
$$ 
Z = \{ x \in M\mid \dim \,\,\ker s_x =1\}
$$
is a smooth submanifold of $M$, of codimension $n-r+1$, equipped with a map 
$Z \to P(\Rr^r)$ to the projective space $P(\Rr^r)$ of
lines in $\Rr^r$ given by the $1$-dimensional kernel. 

The section $s$ defines a section $\sigma$ of $n\lambda \otimes H_r$
over $M \times P(\Rr^r)$ taking $(x,[v])$ to $s_x(v)\otimes v$, 
where $x\in M$ and $v\in S(\Rr^r)\subseteq\Rr^r$ is a unit vector. 
Moreover, this section $\sigma$ is transverse to the zero-section 
and its zero-set is diffeomorphic to $Z$. 
We see, therefore, that the class $[Z\to P(\Rr^r)]$ in 
$\NN_{m+r-1-n}(M \times P(\Rr^r))$ corresponds under duality
to the Euler class 
$$ 
e_{ \NN}(\lambda \otimes H_r)^n \in \NN^n(M \times P(\Rr^r)), 
$$
where $H_r$ denotes the Hopf line bundle over $P(\Rr^r)$
and, in particular, it is independent of the section $s$. 
\begin{defn}
For $1\leq r\leq n$, let $e_{n,r}(M,f)\in\NN_{m+r-1-n}(BO (1)\times P(\Rr^r))$
be the image of $[Z\to P(\Rr^r)]$ under the homomorphism induced by $f$.
\end{defn}
\begin{lem}
The class $e_{n,r}(M,f)$ depends only on the bordism class $[M,f]\in
\NN_m(BO (1))$.
\end{lem}
\begin{proof}
Using the interpretation of $e_{n,r}(M,f)$ in terms
$e_\NN (\lambda\otimes H_r)^n$ (which also allows us to define
$e_{n,r}(M,f)$ without the dimensional restriction: $m< 2(n-r+2)$), 
we can follow the proof of Lemma \ref{bordisminv}.
\end{proof}
Let $\kappa$ be the line bundle over $Z$ with fibre $\kappa_x=\ker s_x$ at 
$x\in Z$.
It is included as a subbundle of the trivial bundle $Z \times \Rr^r$ 
with orthogonal complement $\kappa^\perp$ of dimension $r-1$. We also
have an injective vector bundle map 
$\kappa^\perp = \Rr^r/\kappa \hookrightarrow n \lambda | Z$ 
given by the restriction of $s$ to
$Z$ and so an orthogonal decomposition of $n \lambda\, |\, Z$ as 
$\kappa^\perp \oplus \zeta$, say. The $(n-r+1)$-dimensional normal bundle
$\nu$ of $Z$ in $M$ is identified with $\kappa \otimes \zeta$. 
Then with a careful choice of a tubular neighbourhood $D(\nu)
\hookrightarrow M$ of $Z$, we can identify the restriction of $n \lambda $ 
to $D(\nu)$ with the pullback of $\kappa^\perp \oplus
\zeta$ in such a way that $s$ is given over 
$v \otimes w \,\,(v \in S(\kappa_x), w \in \zeta_x)$ 
in the tubular neighbourhood $D(\kappa \otimes \zeta)$ 
by the map 
$$
\Rr^r = \kappa_x^\perp \oplus \kappa_x \to \kappa_x^\perp \oplus \zeta_x
=\Rr^n\otimes\lambda_x
\quad :\quad  (u,tv)\mapsto (u,tw)\quad (t \in \Rr). 
$$
\begin{prop} \label{Proposition6.1} 
Suppose that $m < 2(n-r+1)$. 
Then a map $(M,f)$ is cobordant in 
$\NN_{m}(BO (1))$ to a map $(M',f')$
such that $n \lambda' = (f')^*(nH)$ has $r$ linearly independent sections 
if and only if 
$$
e_{n,r}(M,f) =0 \in \NN_{m+r-1-n}(BO (1) \times P(\Rr^r)). 
$$
\end{prop}
\begin{proof}
If there is such a representative $(M',f')$, then
$e_{n,r}(M',f')=0$ by construction.
But $e_{n,r}(M,f)=e_{n,r}(M',f')$, because the obstruction depends only
on $[M,f]\in\NN_m(BO (1))$.

Suppose that $e_{n,r}(M,f) = 0$. 
Then we have a compact manifold $V$ with boundary $\partial V=Z$,  
a map 
$g: V \rightarrow BO (1)$ extending $f$ and 
a $1$-dimensional subbundle $\kappa'$ of $V \times \Rr^r$ 
restricting to the subbundle $\kappa$ of $Z \times \Rr^r$ over $Z$.
The 
splitting of $n \lambda \,|\, Z$ as $\kappa^\perp \oplus \zeta$ extends to a 
splitting of $n(g^*H)$ as $(\kappa')^\perp \oplus \zeta'$,
since $\dim V = m+r-n \leq \dim \zeta = n-r+1$.

We use surgery,
as in the proof of Proposition \ref{Proposition4.3},
to construct a manifold 
$$ 
M' = (M - B(\nu) ) \cup_{S(\nu)}S(\kappa' \otimes \zeta').
$$
The section $s$ over $ M - B(\nu)$ is extended from 
$S(\nu) = S(\kappa \otimes \zeta) $ to $S(\kappa' \otimes
\zeta')$ by the same procedure to produce
a section $s'$ of $\Hom (\Rr^r,n\lambda')$ which is
everywhere monomorphic.
\end{proof}
There is a cohomological criterion, too, generalizing Theorem \ref{Theorem5.1}.
\begin{prop}
Let $1\leq r\leq n\leq m$.
Then $e_{n,r}(M,f)=0$ if and only if,
for $q=0,\ldots , r-1$,
$$
\binom{n}{q}
w_{j_1}(M) \cdots w_{j_s}(M)w_1(\lambda )^{m-\sum j_i}[M]=0
$$
for all $j_1,\ldots ,j_s\geq 1$, $s\geq 0$, such that $\sum_{j=1}^s j_i
\leq m-n+q$ {\rm (}that is, $m-\sum j_i\geq n-q${\rm )}. 
\end{prop}
\begin{proof}
Let $\mu$ be the pullback of $H_r$ to $Z$.
The class $e_{n,r}(M,f)$ vanishes if and only if
$$
w_{j_1}(Z)\cdots w_{j_s}(Z) w_1(\lambda )^kw_1(\mu )^l[Z]=0
$$
for all $j_i,\, k,\, l \geq 0$ such that $\sum j_i + k+l = m-n+r-1$.
We make the abbreviation
$$
w_{j_1,\ldots ,j_s}(\xi )=w_{j_1}(\xi )\cdots w_{j_s}(\xi )
$$
for any vector bundle $\xi$.
Now $Z$ embeds in $M\times P(\Rr^r)$ with normal bundle $n(\lambda\otimes H_r)$.
So
$w_{j_1\ldots j_s}(TZ) w_1(\lambda )^kw_1(\mu )^l[Z]=$
$$
w_{j_1,\ldots ,j_s}(TM -n(\lambda\otimes H_r)+rH_r)
w_1(\lambda\otimes H_r)^n w_1(\lambda )^kw_1(H_r)^l[M\times P(\Rr^r)].
$$
As in the proof of Theorem \ref{Theorem5.1}
the  condition for the vanishing of $e_{n,r}(M,f)$ can
be rewritten as:
$$
w_{j_1,\ldots ,j_s}(TM -n(\lambda\otimes H_r)+rH_r)
w_1(\lambda\otimes H_r)^{n+k}w_1(H_r)^l[M\times P(\Rr^r)]=0
$$
for all $j_i,\, k,\, l \geq 0$ with $\sum j_i +k +l=m-n+r-1$
and then as
$$
w_{j_1}(M)\cdots w_{j_s}(M) 
w_1(\lambda\otimes H_r)^{n}w_1(\lambda )^kw_1(H_r)^l[M\times P(\Rr^r)]=0
$$
for all $j_i,\, k,\, l \geq 0$ with $\sum j_i +k +l=m-n+r-1$.
Expanding $w_1(\lambda\otimes H_r)^n=(w_1(\lambda )+w_1(H_r))^n$,
we can write this as
$$
w_{j_1}(M)\cdots w_{j_s}(M) 
\binom{n}{r-1-l} w_1(\lambda )^{m-\sum j_i}[M]=0
$$
for all $j_i,\, k,\, l \geq 0$ with $\sum j_i +k +l=m-n+r-1$.
The statement of the proposition is obtained by
substituting $q$ for $r-1-l$.
\end{proof}
\begin{cor}
If $w_{n-q}(n\lambda )=0$ for $0\leq q <r$, then $e_{n,r}(M,f)=0$.
\end{cor}
\begin{proof}
This follows immediately, because 
$w_{n-q}(n\lambda )=\binom{n}{q}w_1(\lambda )^{n-q}$.
\end{proof}
\begin{rem}
The arguments in this section
may be generalized to replace $BO (1)$ by a general CW-complex
$B$ and $nH$ by an $n$-dimensional real vector bundle $\eta$ over
$B$.  
Given a pair $(M,f)$ consisting
of a closed $m$-manifold $M$ and a map $f: M\to B$, representing
a class in $\NN_m(B)$, 
there is a precise obstruction in $\NN_{m+r-1-n}(B\times P(\Rr^r))$,
coming from the dual of the Euler class $e_\NN (\eta\otimes H_r)$,
which vanishes if and only if $(M,f)$ is cobordant 
to a pair $(M',f')$ such that $(f')^*\eta$ has $r$
linearly independent sections.
\end{rem}

\end{document}